\documentclass[12pt,a4paper]{amsart}

\usepackage[top=3cm, bottom=3cm, left=2.5cm, right=2.5cm]{geometry}
\usepackage{amssymb}
\usepackage{amscd}
\usepackage[latin1]{inputenc}
\usepackage{wrapfig}
\usepackage{psfrag}
\usepackage{epsfig}
\usepackage{epic}
\usepackage{eepic}

\usepackage{hyperref}

\usepackage{caption}

\usepackage{setspace}

\usepackage{tikz}
\usetikzlibrary{shapes, quotes}

\numberwithin{figure}{section}

\numberwithin{equation}{section}
\swapnumbers

\newtheorem{Lemma}[equation]{Lemma}
\newtheorem{Theorem}[equation]{Theorem}
\newtheorem*{thm*}{Theorem}

\newtheorem*{Question*}{Question}

\newtheorem{Proposition}[equation]{Proposition}
\newtheorem{Corollary}[equation]{Corollary}
\newtheorem{Claim}[equation]{Claim}
\newtheorem*{Lemma*}{Lemma}
\newtheorem*{Corollary*}{Corollary}
\theoremstyle{remark}
\newtheorem{Remark}[equation]{Remark}

\theoremstyle{definition}

\newtheorem{noTitle}[equation]{}

\newcommand{\ow}{\omega}

\newcommand{\R}{\mathbb{R}}
\newcommand{\Z}{\mathbb{Z}}
\newcommand{\C}{\mathbb{C}}
\newcommand{\N}{\mathbb{N}}
\newcommand{\B}{\mathbb{B}}
\newcommand{\bP}{\mathbb{P}}

\newcommand{\CP}{{\mathbb C}{\mathbb P}}

\newcommand{\ft}{\mathfrak{t}}

\newcommand{\J}{\mathcal{J}}

\newcommand{\calF}{\mathcal{F}}
\newcommand{\calO}{\mathcal{O}}

\newcommand{\Id}{\mathrm{Id}}

\newcommand{\tsigma}{\tilde{\sigma}}
\newcommand{\tow}{\tilde{\omega}}

\newcommand{\del}{\partial}
\newcommand{\bdel}{\bar{\partial}}

\DeclareMathOperator \Ham {Ham}
\DeclareMathOperator\Symp{Symp}
\DeclareMathOperator \Hirz {Hirz}
\DeclareMathOperator \pt {pt}

\DeclareMathOperator \PSL {PSL}

\newif\ifdebug                                                      %
\debugfalse

\newcommand{\printname}[1]
   {\smash{\makebox[0pt]{\hspace{-1.0in}\raisebox{8pt}{\tiny #1}}}}

\newcommand{\labell}[1] {\label{#1}{\ifdebug{\printname{#1}}\fi}}

\newcommand{\mute}[1] {}
\begin{document}

\title[Cyclic actions on rational ruled symplectic four-manifolds]
{Cyclic actions on rational ruled symplectic four-manifolds 
}

\author[River Chiang]{River Chiang}
\address{Department of Mathematics, National Cheng Kung University, Tainan 701, Taiwan}
\email{riverch@mail.ncku.edu.tw}

\author[Liat Kessler]{Liat Kessler}
\address{Department of Mathematics, Physics, and Computer Science, University of Haifa,
at Oranim, Tivon 36006, Israel}
\email{lkessler@math.haifa.ac.il}

\begin{abstract}
	Let $(M,\omega)$ be a ruled symplectic four-manifold. If $(M, \ow)$ is rational,
	then every homologically trivial symplectic cyclic action on $(M,\omega)$ is the restriction of a Hamiltonian circle action. 
\end{abstract}

\maketitle

\section{Introduction}

\begin{Theorem} \labell{main1}
Let $(M,\omega)$ be a rational ruled symplectic four-manifold. 
Then every homologically trivial symplectic cyclic action on $(M,\omega)$ is the restriction of a Hamiltonian circle action. 
\end{Theorem}

A \emph{ruled symplectic four-manifold} is an $S^2$-bundle over a closed Riemann surface, with a symplectic ruling: a symplectic form on the total space that is nondegenerate on each fiber. It is \emph{rational} if the Riemann surface is $S^2$, and \emph{irrational} otherwise.

A \emph{cyclic} action is an 
effective 
action of a cyclic group $\Z_r = \Z / r \Z$ of finite order $1<r< \infty$.
An action is called \emph{homologically trivial} if it induces the identity map on homology.

An effective symplectic action of a torus $T=T^k=(S^1)^k$ on a symplectic manifold $(M,\omega)$ is \emph{Hamiltonian} if there exists a \emph{moment map}, that is, an invariant smooth map
$\Phi \colon M \to \ft^{*}\cong {\R}^k$  such that 
$d \Phi_j = - \iota(\xi_j) \omega $
for all $j=1,\dots,k$, where $\xi_1,\ldots,\xi_k$ are the vector fields
that generate the torus action.
A Hamiltonian circle action is always homologically trivial because the circle group is connected.

A Hamiltonian action of a torus $T$ defines a one-to-one homomorphism from $T$ to the
group of Hamiltonian diffeomorphisms $\Ham(M,\omega)$.
The image
is a subgroup of $\Ham(M,\omega)$ that is isomorphic to $T$. Every circle subgroup of $\Ham(M,\omega)$ is obtained this way.

\begin{Corollary}
Let $(M,\omega)$ be a rational ruled symplectic four-manifold. Then every finite cyclic subgroup of $\Ham(M,\omega)$ embeds in a circle subgroup of $\Ham(M,\omega)$.
\end{Corollary}

Since every compact $4$-dimensional Hamiltonian $S^1$-space is K\"ahler, i.e., admits a complex structure such that the action is holomorphic and the symplectic form is K\"ahler \cite[Theorem 7.1]{karshon}, Theorem \ref{main1} implies the following corollary.

\begin{Corollary}\labell{cor-holo}
Let $(M,\omega)$ be a ruled symplectic four-manifold equipped with a homologically trivial symplectic cyclic action. If $(M, \ow)$ is rational, there exists an integrable almost complex structure compatible with $\omega$ such that the cyclic action is holomorphic. 
\end{Corollary}

\subsection*{Related works}
This study fits into the broader topological classification of group actions on four-manifolds.
The question of whether a (pseudofree, homologically trivial, locally linear) cyclic action on a four-manifold $M$ must be the restriction of a circle action is studied also in the holomorphic, differentiable, and topological categories; see \cite{Edmonds}. When $M=\CP^2$, the answer is \emph{Yes} in all three categories \cite{Wilc2}, (in fact, these categories coincide).
When $M$ is a rational ruled surface, the answer is \emph{Yes} in the holomorphic category, however
there exist ``exotic" homologically trivial pseudofree cyclic actions on $M$ that are not holomorphic 
but are smoothable with respect to some smooth structure \cite{wilczynski}. By our result, Corollary \ref{cor-holo},
these actions cannot be symplectic.

In the symplectic category, Chen \cite{chen1} showed that the answer is \emph{Yes} when $(M,\omega)$ is $\CP^2$ with the standard Fubini--Study form. Furthermore, 
Chen  \cite{chen1} 
gave an answer in a non-closed case: the answer is \emph{Yes} if $(M,\omega)$ is the four-ball $\B^4$ with the standard symplectic form $\sum_{i=1}^{2}{dx_i \wedge dy_i}$, and the cyclic action is linear near the boundary of $\B^4$.

In a previous paper \cite{counter}, we showed that the statement of Theorem \ref{main1} does not hold in general. On $(M,\omega)$, which is either a six-point blowup of certain sizes of $\CP^2$, or a three-point blowup of certain sizes of an irrational ruled symplectic four-manifold, we gave an example of symplectic action of $\Z_{2}$, 
acting trivially on homology, and showed that it cannot extend to a Hamiltonian circle action; we also showed that $(M,\omega)$ does admit a Hamiltonian circle action.

\subsection*{Acknowledgements}
The question of whether every finite cyclic subgroup of the group of Hamiltonian diffeomorphisms embeds in a circle subgroup
was asked by Jarek Kedra at the ``Hofer 20" conference in Edinburgh in July 2010.
We thank Yael Karshon and Leonid Polterovich for telling us about it, and for helpful suggestions and stimulating discussions.

\section{Almost complex structures and 
pseudo-holomorphic spheres}

In this section we recall definitions and results on almost complex structures and pseudo-holomorphic spheres, and deduce corollaries for pseudo-holomorphic spheres with respect to invariant almost complex structures.

\subsection*{Almost complex structures}

An \emph {almost complex structure} on a manifold $M$ is an automorphism $J \colon TM \to TM$
such that $J^2 = -\Id$. An almost complex structure $J$ is \emph{integrable} if it is induced from a complex
manifold structure. 
It is \emph{compatible}
with $\omega$ if
$
\langle \cdot, \cdot \rangle:=\omega(\cdot,J\cdot)
$
is a Riemannian metric on $M$.
A \emph{K\"ahler} manifold is a symplectic manifold with an integrable almost complex structure that is compatible with the symplectic form.
The space of all compatible almost complex structures on a symplectic manifold $(M,\omega)$ is non-empty and contractible \cite[Proposition 4.1]{MS:intro}. We denote it by $\J:=\J(M,\omega).$ The first Chern class of the complex vector bundle $(TM,J)$
is independent of the choice of $J \in \J$; we denote it by $c_1(TM)$.

\begin{noTitle}\labell{PD}
Let $(V, \ow)$ be a symplectic vector space and $g$ an 
inner product on $V$. Denote by $g^\sharp: V \to V^*$ the isomorphism $u \mapsto g(u, \cdot)$ and by $\ow^\sharp: V \to V^*$ the isomorphism $v \mapsto \ow(v, \cdot)$. Then $A = (g^\sharp)^{-1} \circ \ow^\sharp$ is anti-symmetric with respect to $g$. Moreover, $AA^*$ is symmetric and positive definite. Let $P$ be a positive square root of $AA^*$. Then $J=P^{-1}A$	is a compatible complex structure on $V$. The factorization $A = PJ$ is called the \emph{polar decomposition} of $A$.

Suppose $(M, \ow)$ is a symplectic manifold and $g$ a Riemannian metric on $M$. Then we note that the polar decomposition is canonical and the above construction of a compatible almost complex structure is smooth. See \cite[Sec 3.1]{Ana:book}.
\end{noTitle}

\begin{Claim}\labell{claimginv}
For every symplectic action of a compact Lie group $G$ on a compact symplectic manifold $(M,\omega)$ there exists a $G$-invariant $\omega$-compatible almost complex structure.
\end{Claim}

\begin{proof}
Take a $G$-invariant Riemannian metric $g$ on $M$. Such a metric is obtained from some Riemannian metric $g'$ by averaging with respect to the action of $G$ as follows:
$$
g(u,v):=\int_{G}g'({\sigma_{a}}_{*}u,{\sigma_{a}}_{*}v)\,d a
$$
for $u,v \in TM$, where $G \owns a \mapsto \sigma_a \in \Symp(M, \ow)$ denotes the action and $da$ denotes the Haar measure.
 The  polar decomposition in \S \ref{PD} associated to the invariant metric $g$ provides a $G$-invariant almost complex structure.
 \end{proof}

\subsection*{$J$-holomorphic spheres in 
four-dimensional symplectic manifolds}

A \emph{(parametrized) $J$-holomorphic sphere}, or \emph{$J$-sphere} for short, is a map  from $\CP^1$
to an almost complex manifold, $f \colon (\CP^1,j) \to (M,J)$  that satisfies
the Cauchy--Riemann equations
$
df \circ j = J \circ df
$
at every point in $\CP^1$.
The image $C=f(\CP^1)$ is an \emph{unparamterized $J$-sphere.}
A parametrized $J$-holomorphic sphere is  called \emph {simple} if it cannot be factored through a branched covering of the domain.
An embedding is a one-to-one immersion which is a homeomorphism
with its image.
An \emph{embedded J-sphere} $C \subset M$ is the image
of a $J$-holomorphic embedding $f \colon (\CP^1, j) \to (M, J)$. 
If $J$ is $\omega$-compatible, then such a $C$ is an embedded $\omega$-symplectic sphere.

\emph{Gromov's compactness theorem} \cite[1.5.B]{gromovcurves} guarantees that, given a converging sequence of almost complex structures on a compact manifold, a corresponding sequence of holomorphic curves with bounded symplectic areas has a weakly converging subsequence; the limit under weak convergence might be  a connected union of holomorphic curves. 

In dimension four, $J$-holomorphic spheres admit nice properties, which we will use in this study, such as the adjunction formula \cite[Corollary~E.1.7]{nsmall}, the Hofer--Lizan--Sikorav regularity criterion \cite{HLS}, and the positivity of intersections of $J$-holomorphic spheres \cite[Appendix E and Proposition 2.4.4]{nsmall}. As a result of the positivity of intersections we have the following claim.

\begin{Claim} \labell{ginvjhol}
Consider a symplectic action of a compact Lie group $G$ on a symplectic four-manifold $(M,\omega)$. Assume that the action is trivial on $H_{2}(M;\Z)$. Let $J_G$ be a $G$-invariant $\omega$-compatible almost complex structure on $M$.
Let $C$ be a $J_G$-holomorphic sphere. Then $C$ is $\omega$-symplectic.
Moreover, let $G \to \Symp(M, \ow)$, $g \mapsto \sigma_g$, denote the action,
\begin{itemize}
\item if $C$ is of negative self-intersection, then it is $G$-invariant, i.e., $\sigma_g(C)=C$ for all $g \in G$;
\item if $C$ is of self-intersection zero, then for $g \in G$, the image of $C$ under $\sigma_g$ either equals $C$ or is disjoint from $C$.
\end{itemize}
\end{Claim}

\section{
Configurations of $J$-holomorphic spheres in 
rational ruled symplectic four-manifolds
}

\begin{noTitle} The underlying space of a rational ruled symplectic four-manifold is either $S^2 \times S^2$ (the trivial bundle) or $\CP^2 \# \overline{\CP^2}$ (the non-trivial bundle). 

In $H_2(S^2 \times S^2;\Z)$, denote the homology classes $B:= [S^2 \times \pt]$ and $F:=[\pt \times S^2]$.

In $H_2(\CP^2 \# \overline{\CP^2};\Z)$, denote by $L$ the homology class of a line $\CP^1$ in $\CP^2$ and by $E$ the class of the exceptional divisor;
denote
 by $F$ the fiber class of the fibration $\CP^2 \# \overline{\CP^2} \to S^2$. Note that
$F=L-E.$

For $\lambda \geq 0$, denote by 
$\omega_{\lambda}^{0}$ the symplectic form $(1+\lambda) \tau \oplus \tau$ on $S^2 \times S^2$ 
where $\tau$ is the standard area form on $S^2$ with $\int_{S^2}\tau=1$.

For $\lambda >-1$, denote by 
$\omega_{\lambda}^{1}$ 
the symplectic form on $\CP^2 \# \overline{\CP^2}$ 
for which there exist an embedded symplectic sphere of area $(2+\lambda)$ in the class $L$ and an embedded symplectic sphere of area $(1+\lambda)$ in the class $E$.

By the work of Gromov \cite{gromovcurves}, Li--Liu \cite{LL2}, Lalonde--McDuff \cite{ML1},  McDuff \cite{McDuff-Structure} and Taubes \cite{Taubes}, up to scaling, every symplectic form on $S^2 \times S^2$ is of the form $\omega_{\lambda}^0$ for $\lambda \geq 0$; every symplectic form on $\CP^2 \# \overline{\CP^2}$ is of the form $\omega_{\lambda}^1$ for $\lambda > -1$.
See also  \cite[Examples 3.5, 3.7]{salamon}.
\end{noTitle}

\begin{noTitle}\labell{strat0}
Consider $(M,\omega)=(S^2 \times S^2, \omega^0_{\lambda}:=(1+\lambda) \tau \oplus \tau)$ with $\lambda \geq 0$.
Take $\ell = \lceil \lambda \rceil$, namely, the integer satisfying
$
\ell-1 < \lambda \leq \ell.
$

Let $J \in \J$.
Note that if a class $A=a_B B+a_F F\in H_{2}(S^2 \times S^2;\Z)$ can be represented by a simple $J$-holomorphic sphere, then, by the adjunction formula, 
\[
0 \leq A\cdot A-c_{1}(A)+2 = 2a_B a_F-2a_B-2a_F+2 =2 (a_B-1)(a_F-1).
\] Also, since $J$ is $\omega$-compatible, $0<\omega(A)=a_B(1+\lambda)+a_F$.
Hence either $a_B,a_F \geq 2$; or $a_B=1$ and $a_F>-(1+\lambda)$; or $a_B \geq 0$ and $a_F=1$;
see \cite[Lemma 1.7]{abreu}.

If $F=\sum_{i=1}^{n} A_i=\sum_{i=1}^{n}(a_B^i B+a_F^i F)$ and each $A_i$ is represented by a simple $J$-holomorphic sphere, then, since $\sum_{i=1}^{n}a_B^i=0$, we must have $a_B^i=0$ and $a_F^i=1$ for all $i$, so $1=\sum_{i=1}^{n} a_F^i=n$.
We conclude that for every $J \in \J$ the class $F$ is \emph{$J$-indecomposable}, i.e., it cannot be written as a sum of two or more classes that are represented by non-constant $J$-holomorphic spheres. However, if $\lambda>0$ then $B=(B-F)+F$, 
with $\int_{(B-F)}\omega^0_{\lambda}=\lambda>0$, $\int_{F}\omega^0_\lambda=1>0
$, and $B-F$ is the homology class of the symplectically embedded antidiagonal $\{(s,-s) \in S^2 \times S^2\}$. 

The following facts are derived form Gromov's compactness theorem and the properties of $J$-holomorphic spheres in dimension four, see Gromov \cite{gromovcurves}, Abreu \cite{abreu}, and Abreu--McDuff \cite{AbMc}.

\begin{enumerate}
\item
The space $\J^{0}_{\lambda}:=\J(M,\omega)=\J(S^2 \times S^2,\omega_{\lambda}^{0})$ is stratified as follows:
\begin{equation} \labell{eqstrat}
\J(S^2 \times S^2,\omega_{\lambda}^{0})=U_{0}^{0} \cup U_{1}^{0} \cup \cdots U_{\ell}^{0}, \quad \ell =\lceil \lambda \rceil \in \N,
\end{equation}
where
\begin{multline*}
	\qquad\qquad\qquad U_k^{0}:=\left\{ J \in \J^0_{\lambda} \mid B-kF \in H_{2}(S^2 \times S^2;\Z)\right.\\ 
	\left. \text{ is represented by a simple }J\text{-holomorphic sphere}\right\}.
\end{multline*}	
In particular,  $U_0^{0}$ is open and dense in $\J_{\lambda}^{0}$.  
\item For every $J \in \J(S^2 \times S^2,\omega_{\lambda}^{0})$, for any $p \in S^2 \times S^2$ there is an embedded $J$-holomorphic sphere in the class $F$ passing through the point $p$, unique up to reparametrizations. The family of these spheres forms the fibers of a fibration $S^2 \times S^2 \to S^2$.
\item For $J \in U_0^0$, the previous statement holds also in the class $B$ instead of $F$. Hence for such $J$ there are two foliations $\calF^{B}_{J}$ and $\calF^{F}_{J}$ whose leaves are embedded $J$-holomorphic spheres in $B$ and in $F$, respectively; the leaf in $\calF^{B}_{J}$ (resp. $\calF^{F}_{J}$) through a point $p$ is the unique $J$-holomorphic sphere in $B$ (resp. $F$) through $p$; any two spheres in the same foliation are disjoint; each sphere in  $\calF^{B}_{J}$  intersects each sphere in $\calF^{F}_{J}$  at exactly one point and  transversally.
\item Moreover, for $J \in U_0^0$, there is a diffeomorphism $\Psi_{J} \colon S^2 \times S^2 \to M$ such that
$\Psi_{J}$ maps the $J_0$-foliations $\calF^B_{J_0}=\{S^2 \times \pt\}$, $\calF^F_{J_0} =\{\pt \times S^2\}$ to the corresponding $J$-foliations, where $\pt \in S^2$ and $J_0=j\times j$ is the standard split complex structure on $S^2\times S^2$. The symplectic form
$\omega_{J}:=\Psi_{J}^{*}(\omega)$
is cohomologous and 
 linearly isotopic to ${\omega}_{0}=\omega_{\lambda}^{0}$. See the construction in  \cite[Theorem~9.4.7]{nsmall}. Hence, by Moser's lemma,
there is  a diffeomorphism $h \colon S^2 \times S^2 \to S^2 \times S^2$ such that $\Psi_{J} \circ h \colon (S^2 \times S^2, \omega_{0}) \to (M,\omega)$ is a symplectomorphism
that induces the identity map on $H_{2}(S^2 \times S^2;\Z)$.
 \end{enumerate}
\end{noTitle}

\begin{noTitle} \labell{generalfoliation}
Items (3) and (4) above generalize to any compact connected symplectic four-manifold $(M,\omega)$ and $\check{J} \in \J(M,\omega)$ that satisfy the following conditions:
\begin{itemize}
\item There is no symplectically embedded $2$-sphere with self-intersection number $-1$.
\item There exist classes $B, F$  in $H_{2}(M;\Z)$ such that $B \cdot B=F\cdot F=0$, $B \cdot F=1$, and both $B$ and $F$ are represented by embedded $\check{J}$-holomorphic spheres.
\end{itemize}
In such a case, replace $J \in U_0^0$ in items (3) and (4) by $\check{J}$, and replace the form $\omega_0$ in item (4) by $ \check{b}\tau \oplus \check{f}\tau$ where $\check{b}:=\int_{B}\omega$ and $\check{f}:=\int_{F}\omega$. 
See the proof of \cite[Theorem~9.4.7]{nsmall}.
\end{noTitle}

\begin{noTitle}\labell{strat01}
Consider $(M,\omega)=(\CP^2 \# \overline{\CP^2}, \omega^1_{\lambda})$ with $\lambda > -1$. 
Take $\ell = \lceil \lambda \rceil$.
Abreu--McDuff \cite{AbMc} gave an analogous description of $\J(\CP^2 \# \overline{\CP^2},\omega_{\lambda}^{1})$.

\begin{enumerate}
\item
The space $\J^{1}_{\lambda}:=\J(\CP^2 \# \overline{\CP^2},\omega_{\lambda}^{1})$ is stratified as follows:
\begin{equation} \labell{eqstrat2}
\J(\CP^2 \# \overline{\CP^2},\omega_{\lambda}^{1})=U_{0}^{1} \cup U_{1}^{1} \cup \cdots U_{\ell}^{1}, \quad \ell =\lceil \lambda \rceil \in \N,
\end{equation}
where
\begin{multline*}
\qquad\qquad\qquad U_k^{1}:=\left\{J \in \J^1_{\lambda} \mid E-kF \in H_{2}(\CP^2 \# \overline{\CP^2};\Z)\right.\\ \left.\text{ is represented by a simple }J\text{-holomorphic sphere}\right\}.
\end{multline*}
 \item For every $J \in \J(\CP^2 \# \overline{\CP^2},\omega_{\lambda}^{1})$, the fiber class $F$ is represented by a $2$-parameter family of embedded $J$-holomorphic spheres that fiber $\CP^2 \# \overline{\CP^2}$.
 \end{enumerate}
\end{noTitle}

\section{Proof of the main theorem}

  \begin{Proposition}\labell{propequal}
Let $(M,\omega)=(S^2 \times S^2,\omega_{\lambda}^{0}:=(1+\lambda)\tau \oplus  \tau)$ with $\lambda \geq 0$ equipped with a homologically trivial symplectic action of a finite cyclic group $G$. Let $J=J_G$ be a $G$-invariant $\omega$-compatible almost complex structure on $M$. Assume that $J$  is in the stratum $U_0^0$.
Then the $G$-action is symplectically conjugate to the restriction of a standard circle action on $(M,\omega)$.
 \end{Proposition}

Note that if $\lambda=0$, the space $\J=U_0^0$. Hence the proposition implies Theorem \ref{main1} for the case $(S^2 \times S^2,\tau \oplus  \tau)$.

We call a symplectic circle action on $(M,\omega)=(S^2 \times S^2,\omega_{\lambda}^{0})$  standard if the circle rotates the first sphere at speed $a$ and the second at speed $b$, where $a$ and $b$ are relatively prime integers.

\begin{proof}
Fix a generator $g \in G$ and denote by $\sigma_g$ the corresponding symplectomorphism defined by the action, i.e., $\sigma_g$ is the image of $g$ under the homomorphism $G \to \Symp(M,\omega)$. 
Since $J$ is $G$-invariant and $\sigma_g$ induces the identity map on $H_2(S^2 \times S^2; \Z)$, the map $\sigma_g$ sends a $J$-holomorphic sphere representing a class in $H_2(S^2 \times S^2;\Z)$ to a $J$-holomorphic sphere representing the same class.
Since $J \in U_0^0$, items (3) and (4) in \S \ref{strat0} apply; in particular there are foliations $\calF^{B}_{J}$ and $\calF^{F}_{J}$ of embedded $J$-holomorphic spheres; the leaves of the foliations through a given point are unique. Therefore, $\sigma_g$ sends each leaf of  $\calF^{B}_{J}$ (resp. $\calF^{F}_{J}$) to a leaf of  $\calF^{B}_{J}$  (resp. $\calF^{F}_{J}$).

Let $J_0 = j \times j$ be the standard split complex structure on $S^2 \times S^2$. By item (4) in \S \ref{strat0}, there exists a symplectomorphism $\Psi_J \colon (S^2 \times S^2, \ow_J=\Psi^*_J\ow) \to(S^2\times S^2, \ow) $ that maps the $J_0$-foliations, $\calF^{B}_{J_0}$, $\calF^{B}_{J_0}$, to the corresponding $J$-foliations, $\calF^{B}_{J}$, $\calF^{F}_{J}$.
Note that the leaves of the $J_0$-foliations coincide with the leaves of the $\Psi_J^*J$-foliations.

Now, define a map $\sigma_1: S^2 \to S^2$ by the following procedure. For $x \in S^2$,  $v_x = \{x\}\times S^2$ is a leaf in $\calF^{F}_{J_0}$; namely, a $J_0$-holomorphic sphere representing $F$.
Then $\Psi^{-1}_J\sigma_g \Psi_J (v_x)$ is a leaf in $\calF^{F}_{J_0}$, and therefore of the form $\{x'\} \times S^2$. Define $\sigma_1(x) = x'$.
Similarly we define $\sigma_2: S^2 \to S^2$ by comparing $u_y = S^2 \times \{y\}$ and
$\Psi^{-1}_J\sigma_g \Psi_J (u_y)$. Note that each $\sigma_i$ generates an action of
$G$ on $S^2$ because $(\Psi^{-1}_J\sigma_g \Psi_J)^r = \Psi^{-1}_J(\sigma_g)^r \Psi_J = \Psi^{-1}_J\sigma_{g^r} \Psi_J$. Observe that $\sigma_i$ is orientation preserving with respect to the orientation defined by $\ow_J$.

We check that $\Psi_J$ is $G$-equivariant with respect to the split action generated by $\sigma_1 \times \sigma_2$ and the given $G$-action, i.e., $$\Psi^{-1}_{J}\sigma_g \Psi_{J}(x,y)=(\sigma_1 \times \sigma_2) (x,y) \text{ for }(x,y) \in S^2 \times S^2.$$
The image $p=\Psi_J(x, y)$ is the intersection
of the $J$-holomorphic spheres
$\Psi_J(v_x)$ and $\Psi_J(u_y)$. By item (3) in \S \ref{strat0}, they are the unique $J$-holomorphic spheres passing through $p$ in the classes $F$ and $B$ respectively. By the same token, $\sigma_g\Psi_J(v_x)$
and $\sigma_g\Psi_J(u_y)$ are the unique $J$-holomorphic spheres passing through $\sigma_g(p)$ in $F$ and $B$, and $\Psi_J^{-1}\sigma_g\Psi_J(v_x)$
and $\Psi_J^{-1}\sigma_g\Psi_J(u_y)$ the unique $J_0$-holomorphic spheres passing through
$\Psi^{-1}_J\sigma_g\Psi_J(x, y)$.  In other words, $\Psi^{-1}_J\sigma_g\Psi_J(x, y)$ is the intersection of $\Psi_J^{-1}\sigma_g\Psi_J(v_x)$ and $\Psi_J^{-1}\sigma_g\Psi_J(u_y)$, which is $(x', y')=(\sigma_1(x), \sigma_2(y)) =:(\sigma_1 \times \sigma_2)(x,y)$. We have the following commutative diagram of symplectomorphisms:

\begin{equation}\labell{diagram-1}
	\begin{CD}
		(S^2 \times S^2, \ow_J=\Psi^*_J\ow) @>{\Psi_J}>> (S^2\times S^2, \ow)\\
		@V{\sigma_1 \times \sigma_2}VV 		@VV{\sigma_g}V \\
		(S^2 \times S^2, \ow_J=\Psi^*_J\ow) @>{\Psi_J}>> (S^2\times S^2, \ow)
\end{CD}\end{equation}

By Proposition \ref{moseractst},
the cyclic action on $S^2$ generated by $\sigma_i:S^2 \to S^2$ is conjugate to the restriction of a circle action on $S^2$ by an orientation-preserving
diffeomorphism $h_i$.
That is,
we have the following commutative diagram
\begin{equation}\labell{diagram-2}
	\begin{CD}
		(S^2 \times S^2, \tow_J=(h_1\times h_2)^*\ow_J) @>{h_1 \times h_2}>> (S^2 \times S^2, \ow_J)\\
		@V{\tsigma_1\times \tsigma_2}VV    @VV{\sigma_1 \times \sigma_2}V 		 \\
		(S^2 \times S^2, \tow_J=(h_1\times h_2)^*\ow_J) @>{h_1 \times h_2}>> (S^2 \times S^2, \ow_J)
\end{CD}\end{equation}
where $\tsigma_i:S^2 \to S^2$ denotes a generator of the restriction of a circle action on $S^2$ to the cyclic subgroup. More precisely, if $a_i$ is the rotation number of $\sigma_i$ on $S^2$, and $r$ stands for the order of $G$, then $\gcd(a_1,a_2,r)=1$ because the given $G$-action is effective. Let $\gcd(a_1,a_2)=c$. Then $\gcd(c, r)=1$, $\gcd(a_1/c, a_2/c)=1$, and $\tsigma_i$ is the restriction induced by the inclusion $G \hookrightarrow S^1$, $g \mapsto g^c$ of the circle action on $S^2$ by rotation at speed $a_i/c$.  
Note that $h_1 \times h_2$ preserves the $J_0$-foliations.

Since $\Psi_J$ and $h_1 \times h_2$ send the class $B$ to $B$ and the class $F$ to $F$, $\tow_J$ and $\ow_0=\ow_\lambda^0$ are cohomologous. 
Since both $\ow_0$ and $\tow_J$ tame $(h_1\times h_2)^*\Psi_J^*J$ and are $G$-invariant with respect to the action generated by $\tsigma_1 \times \tsigma_2$, the $2$-form
\[
\ow_t = (1-t)\ow_0 + t\tow_J
\] is symplectic and $G$-invariant for all $t \in [0,1]$. By Moser's method, one can find a $G$-equivariant
vector field whose time-one flow gives a $G$-equivariant diffeomorphism $h:S^2 \times S^2 \to S^2 \times S^2$ such that $h^*\tow_J =\ow_0$. Namely,
\begin{equation}\labell{diagram-3}
	\begin{CD}
		(S^2 \times S^2, \ow_0) @>{h}>> (S^2 \times S^2, \tow_J)\\
		@V{\tsigma_1\times \tsigma_2}VV    @VV{\tsigma_1 \times \tsigma_2}V 		 \\
		(S^2 \times S^2, \ow_0) @>{h}>> (S^2 \times S^2, \tow_J)
\end{CD}\end{equation}

Combining these three diagrams \eqref{diagram-1}--\eqref{diagram-3}, we conclude that the cyclic group action is symplectically conjugate to 
the restriction of a standard circle action.
\end{proof}

\begin{Remark} \labell{remequal}
If $J_G \in \J(S^2\times S^2,\omega_{\lambda}^{0})$ is $G$-invariant and for some $x,y \in S^2$, the spheres $\{x\} \times S^2$ and $S^2 \times \{y\}$ are $G$-invariant and $J_G$-holomorphic,  then, by the definition of $U_{0}^{0}$, $J_G \in U_{0}^{0}$ and $G$ induces the identity on homology. Hence Proposition \ref{propequal} applies.
Moreover, by the proof of Proposition \ref{propequal}, there is a symplectomorphism $\psi:S^2 \times S^2 \to S^2 \times S^2$ that conjugates the $G$-action and
a standard action on $S^2 \times S^2$. If the $G$-action on $\{x\} \times S^2$ and $S^2 \times \{y\}$ is standard, then $\psi$ can be chosen to be the identity on these spheres. 

In fact this holds for any symplectic form $\ow$ on $S^2\times S^2$ instead of $\omega_{\lambda}^{0}$, by \S \ref{generalfoliation} and the fact that every symplectic sphere in $S^2\times S^2$ has even self-intersection due to homological reason.   
\hfill $\diamondsuit$
\end{Remark}

Next we show an equivariant version of Gromov's result \cite{gromovcurves} on $D^2 \times D^2$, as appeared in \cite[Theorem~1.9]{abreu}. 

Here $D^2$ denotes the unit disc in $\R^2$. Let $\ow_\lambda = (1+\lambda)dx^1 \wedge dx^2 +dy^1 \wedge dy^2$ be a split symplectic form where $(x, y) = ((x^1, x^2), (y^1, y^2)) \in \R^2 \times \R^2$. We call an action of a finite cyclic group $G$ on $\R^2 \times \R^2 \cong \C \times \C$ \emph{linear} if $\sigma_g(x,y)=(\xi^{a}x,\xi^{b}y)$, where $g$ is a generator of $G$, $\sigma_g$ is the corresponding map defined by the action, and $\xi$ is the corresponding root of unity.

\begin{Lemma} \labell{d2lemma}
Let  a cyclic group $G$ of finite order act symplectically on the product of unit discs $D^2 \times D^2$ with a symplectic form $\ow$. 
Assume that $\omega$ equals $\ow_\lambda$ near the boundary, and 
the $G$-action is linear near the boundary.

Then the given $G$-action on $D^2 \times D^2$ is conjugate to the linear action by a symplectomorphism that is the identity on the boundary.
\end{Lemma}

\begin{proof}
Consider the collapsing map $$\kappa \colon D^2 \times D^2 \to S^2 \times S^2$$ that identifies the points of $\{x\} \times \del D^2$, for $x \in D^2$, as $(x,\infty)$, and identifies the points of $\del D^2 \times \{y\}$, for $y \in D^2$, as $(\infty,y)$. 
Let $J_G$ be a $G$-invariant $\omega$-compatible almost complex structure on $D^2 \times D^2$ that is equal to the standard split complex structure $J_0$ near the boundary. Such $J_G$ exists by taking any $\ow$-compatible almost complex structure $J$ that is equal to $J_0$ near the boundary, averaging the metric $\omega(\cdot,J \cdot)$ with respect to the $G$-action, and then taking the structure provided by the polar decomposition, as in Claim \ref{claimginv}. 

Consider the image $S^2 \times S^2$ with the induced symplectic form and the induced (compatible) almost complex structure, again denoted by $\ow$ and $J_G$. 
Let $g \in G$ be a generator of $G$ and $\sigma_g$ be the corresponding symplectomorphism of $D^2\times D^2$ defined by the action. Denote by $\bar{\sigma}_g$ the symplectomorphism of $S^2 \times S^2$ that pulls back to $\sigma_g$. On the image of the boundary,  $\bar{\sigma}_g(x,\infty)=(\xi^{a}x,\infty)$ and $\bar{\sigma}_g(\infty, y)=(\infty, \xi^{b}y)$. On the complement of $S^2 \times \{\infty\} \cup \{\infty\} \times S^2$, we have $\bar{\sigma}_g(x,y)=\sigma_g(\kappa^{-1}(x,y))$. 

The spheres $S^2 \times \{\infty\}$ and $\{\infty\} \times S^2$ in $S^2 \times S^2$ are $J_G$-holomorphic and invariant under $\bar{\sigma}_g$. 
Moreover, the $G$-action $\bar{\sigma}_g$ is standard on these spheres. As explained in Remark \ref{remequal}, by Proposition \ref{propequal}, there is a symplectomorphism $\psi \colon S^2 \times S^2 \to S^2 \times S^2$ conjugating $\bar{\sigma}_g$ and a standard action, and $\psi$ is the identity on $\{\infty\} \times S^2$ and $S^2 \times \{\infty\}$. Hence this standard action is the one induced from $(x, y) \mapsto (\xi^a x, \xi^b y)$.

Define  $$\Psi \colon D^2 \times D^2 \to D^2 \times D^2$$ as the identity on $ D^2 \times \del D^2 \cup \del D^2 \times D^2$ and $\kappa^{-1}\circ \psi \circ \kappa$ on the complement of the boundary.
 Then $\Psi$ is a symplectomorphism conjugating the given action and the linear action, that is the identity on the boundary.
\end{proof}

We now prove Theorem \ref{main1} for the cases that are not covered by Proposition \ref{propequal}, i.e., if
$(M,\omega)=(S^2\times S^2,\omega_{\lambda}^{0})$ and $J=J_G$ is not in $U_0^0$  or if $(M,\omega)=(\CP^2 \# \overline{\CP^2},\omega_{\lambda}^{1})$.
A $G$-invariant $\omega$-compatible almost complex structure $J_G$ on $(M,\omega)$ exists by Claim \ref{claimginv}.

\begin{Proposition}\labell{no-hirz}
Let a cyclic group $G$ of finite order act symplectically on $(M^{\star}_{\lambda},\omega^{\star}_{\lambda})$ for $\star=0, \,1$. Assume that the action is trivial on $H_2(M^{\star}_{\lambda};\Z)$.
 Let $J=J_G$ be a $G$-invariant $\omega^{\star}_{\lambda}$-compatible almost complex structure.
Set $A_k:=B-kF$ if $\star=0$ and $A_k:=E-kF$ if $\star=1$.

Assume that $A_k$ is represented by a simple $J$-holomorphic sphere $C_1$ for an integer $k > -\star$.
 Then the following hold.
\begin{enumerate}
\item The  sphere $C_1$ is embedded and  $G$-invariant and there is an embedded $G$-invariant sphere $C_2$ in the class $F$ such that $C_1$ and $C_2$ intersect at a point; the $G$-action  on each sphere is, up to conjugation by a symplectomorphism of the sphere, a rotation, with rotation numbers $(a, -b)$ at $C_1 \cap C_2$.
\item The $G$-action is symplectically isomorphic to the standard $G$-action on the Hirzebruch surface $\Hirz_{n}(a,-b)$ with the symplectic form $\eta_{\mu}$, as described in  \S \ref{hirzsympl} and \S \ref{hirzcyclic}, where $n=2k+\star$ and $\mu=1+\lambda-k$. 
\end{enumerate} 
\end{Proposition}

The proof of item (2) is an equivariant variation of the proof of \cite[Lemma 3.5]{abreu}.

\begin{proof}\hfill
\begin{enumerate}
\item Fix a generator $g \in G$ and denote by $\sigma_g$ the corresponding symplectomorphism defined by the action. By Claim \ref{ginvjhol}, since the self-intersection number of $A_k$ is negative, the $J$-holomorphic sphere $C_1$ in $A_k$ is $G$-invariant. By the adjunction formula \cite[Corollary~E.1.7]{nsmall} and since $A_k \cdot A_k-c_{1}(A_k)+2=0$, the simple $J$-holomorphic sphere $C_1$ is embedded.
By Proposition \ref{moseractst}, the restriction of the $G$-action to $C_1$ is conjugate 
to the restriction to a cyclic subgroup of a circle action, i.e., a rotation.
In particular, it fixes two points on $C_1$. Let $p$ be one of these $G$-fixed points with rotation number $a$ at $p$. By item (2) in \S \ref{strat0} and \S \ref{strat01}, there is an embedded $J$-holomorphic sphere $C_2$ in the class $F$ passing through the point $p$. By Claim \ref{ginvjhol}, since $F \cdot F=0$, the image $\sigma_g(C_2)$ either coincides with $C_2$ or is disjoint from it. However, since the $G$-fixed point $p$ is on $C_2$, $\sigma_g(C_2)$ must coincide with $C_2$. So $C_2$ is also invariant under the $G$-action. Hence, by Proposition \ref{moseractst}, up to conjugation by a symplectomorphism, the restriction of the $G$-action to $C_2$ is a rotation with rotation number $-b$ at $p$.

\item
By item (1), $C_1$ and $C_2$ are $G$-invariant symplectic spheres intersecting transversally at one point $p$ with rotation numbers $(a, -b)$ at $p$. Since $G$ is finite abelian, the tangent space at the fixed point $p$ decomposes into a direct sum of symplectically orthogonal eigenspaces of $G$; hence $C_1$ and $C_2$ in fact intersect  $\ow_\lambda^\star$-orthogonally at $p$. The sphere $C_1$ has self-intersection number $-n$ and size $\mu$; the sphere $C_2$ has self-intersection number $0$ and size $1$. Using the equivariant version of Weinstein's symplectic neighbourhood theorem for such pair of submanifolds \cite{guadagni} (also see \cite{mcrae}), we can construct a
diffeomorphism
\[
\alpha \colon \Hirz_{n} \to  M^{\star}_{\lambda}
\]
that sends the zero section $S_0$ to $C_1$ and the fiber at zero $F_0$ to $C_2$
and is
an equivariant symplectomorphism from a neighbourhood $U$ of $S_0 \cup F_0$ in $(\Hirz_{n}(a,-b),\eta_\mu)$ onto a neighbourhood $U'$ of $C_1 \cup C_2$ in $ (M^{\star}_{\lambda},\omega^{\star}_{\lambda})$ with the given $G$-action.
 (See \S \ref{hirzsympl} and \S \ref{hirzcyclic} for the notations.) So it is enough to show that the $\alpha$-pull back of the $G$-action is the standard $G$-action on $\Hirz_{n}(a,-b)$, up to conjugation by a symplectomorphism.

Since $S_0$ and $F_0$ are invariant symplectic spheres, the restriction of the $\alpha$-pull back of the $G$-action to the complement of $S_0 \cup F_0$ in $\Hirz_{n}$ is a symplectic action with respect to $\alpha^*\ow_\lambda^\star$. Since $\alpha$ is a symplectomorphism near $S_0 \cup F_0$, $\alpha^*\ow$ equals $\eta_\mu$ there.

The complement of $S_0 \cup F_0$ in $\Hirz_{n}$ is
\[
\{([w, 1],[z,w^n,1])\}\subset \CP^1 \times \CP^2
\]
and the map $$\beta \colon \C^2 \to \Hirz_{n}$$ given by $(w,z) \mapsto ([w, 1],[z,w^n,1])$ pulls back $\eta_\mu$ to
\[
\begin{split}
\Omega_\mu&=\frac{i}{2\pi}[\mu\del\bdel\log (1+\|w\|^2) + \del\bdel \log (1+\|w\|^{2n} + \|z\|^2)]\\
&=\frac{i}{2\pi}\left( \left[\frac{\mu}{(1+\|w\|^2)^2} + \frac{n^2\|w\|^{2(n-1)}(1+\|z\|^2)}{(1+\|w\|^{2n}+\|z\|^2)^2}\right] dw\wedge d\bar{w}\right. \\
&\phantom{=\frac{i}{2\pi}\left( \left[\frac{\mu}{(1+\|w\|^2)^2}\right.\right. }\qquad\qquad\left. + \frac{(1+\|w\|^{2n})}{(1+\|w\|^{2n}+\|z\|^2)^2} dz\wedge d\bar{z}\right).
\end{split}
\]
A neighbourhood of $\infty$ at $\C^2$ is sent to a neighbourhood of $S_0 \cup F_0$.
The standard action on $\Hirz_n(a, -b)$ pulls back to a linear action, $(w, z) \mapsto (\xi^{-a}w, \xi^{b-na} z)$, and therefore the
$\alpha$-pull back of the $G$-action on $\Hirz_n$ pulls back to an action on $\C^2$ which is linear near infinity.

Denote $$f=\log \bigr((1+\|w\|^2)^{\mu} (1+\|w\|^{2n} + \|z\|^2)\bigr).$$
Note that $\Omega_{\mu}=\frac{i}{2\pi}\del\bdel f$.
Consider the gradient vector field $X_f$ of $f$ with respect to the Riemannian metric induced by $\Omega_\mu$ and the standard almost complex structure $J_0$ on $\C^2$. Then $$L_{X_f}{\Omega_\mu}={\Omega_\mu}.$$
The function $f$ strictly increases along rays from the origin. Moreover, the function $f$ and the Riemannian metric are invariant under any linear $G$-action on $\C^2$. Hence the backward flow of $X_f$ commutes with linear $G$-actions, and 
takes the $(\beta \circ \alpha)$-pull back of the $G$-action on $\C^2$, which is linear near infinity, to an action on $D^2\times D^2$ which is linear near the boundary.  
Now the proposition follows from Lemma \ref{d2lemma}. 

Note that the sphere $S_0$ is the preimage of the top edge of $H_{n,\mu}$ under a moment map of the Hamiltonian $T^2$-action \eqref{toric-act} on $(\Hirz_{n},\eta_\mu)$; the sphere $F_0$ is the preimage of the left edge; see Figure \ref{trapezoid}.
The complement of $S_0 \cup F_0$ is the preimage of the complement of the top and left edges, so it corresponds to a standard star-shaped region in $\C^2$, and the Liouville vector field $X_f$ is a lift of the radial vector on the trapezoid in $\R^2$ emitting from the bottom right corner.

\end{enumerate}
\end{proof}

\appendix

\section{Toric and cyclic actions on Hirzebruch surfaces}
\subsection{Hirzebruch surface}\label{hirzsympl}
Let $\calO(n)$ denote the complex line bundle over $\CP^1$ with first Chern class $n$. The
 \emph{Hirzebruch surface} $\Hirz_n$ is the complex manifold defined as a projective bundle, for any integer $n$:
 \[ \Hirz_n \cong \bP(\calO(0)\oplus \calO(-n))\cong ((\calO(0) \oplus \calO(-n)) \smallsetminus 0)/\C^\times.
 \]
It is a $\CP^1$-bundle over $\CP^1$ obtained by gluing two copies $V_1, V_2$ of the form $\C \times \CP^1_{\text{fiber}}$ along
\[
(\C\smallsetminus \{0\}) \times \CP^1 \to (\C\smallsetminus \{0\}) \times \CP^1 \qquad (z, [x_1, x_2]) \mapsto (\frac{1}{z}, [x_1, z^n x_2]).
\]
The zero section $S_0$ corresponds to $\{(z, [1, 0])\}$; the section at infinity $S_\infty$ corresponds to $\{(z, [0, 1])\}$; the fiber at zero $F_0$ refers to the fiber at $[1, 0]\in \CP^1_{\text{base}}$, namely, $\{(0, [x_1, x_2])\}$ in $V_1$; and the fiber at infinity $F_\infty$ refers to the fiber at $[0,1] \in \CP^1_{\text{base}}$, namely, $\{0, [x_1, x_2]\}$ in $V_2$. The self-intersection numbers of $S_0$, $S_\infty$, $F_0$, $F_\infty$ are $-n,n,0,0$, respectively.

One can equip the Hirzebruch surface $\Hirz_n$ with a symplectic structure through symplectic cutting. Alternatively, 
note $\calO(-1)$ is the tautological line bundle over $\CP^1$, whose fiber at $l \in \CP^1$ is the line $l$ in $\C^2$. Since $\calO(-n)$ is the $n$-th tensor power of $\calO(-1)$,
the Hirzebruch surface can be identified with an algebraic submanifold of
$\CP^1 \times \CP^2$ defined in homogeneous coordinates by
\[
\{([w_1,w_2], [z_0,z_1,z_2]) \in \CP^1 \times \CP^2 \mid w_{1}^{n}z_2=w_{2}^{n}z_1\}.
\]
The isomorphism is given by
\begin{align*}
\text{If }{w_1 \neq 0 }, && ([w_1, w_2], [z_0,z_1,z_2]) &\mapsto (\frac{w_2}{w_1}, [z_0,z_1]) \in V_1  \\
\text{If }{w_2 \neq 0 }, && ([w_1, w_2], [z_0,z_1,z_2]) &\mapsto (\frac{w_1}{w_2}, [z_0,z_2]) \in V_2
\end{align*}
Any K\"ahler form on $\CP^1 \times \CP^2$ restricts to a K\"ahler form on $\Hirz_n$.
For $\mu \geq 0$, we denote by $\eta_\mu$ the sum of the Fubini--Study form on $\CP^1$ multiplied by $\mu$ and the Fubini--Study form on $\CP^2$. We shall use the same notation for its restriction to the Hirzebruch surface.

With respect to the symplectic form $\eta_\mu$, we find that the symplectic spheres
\begin{align*}
S_0&=\{([w_1,w_2], [1,0,0])\} \text{ has size }\mu& &\text{(zero section),}\\
S_\infty&=\{([w_1,w_2], [0, w_1^n, w_2^n])\} \text{ has size }\mu+n& &\text{(section at infinity),}\\
F_0&=\{([1, 0], [z_0,z_1,0])\} \text{ has size }1& &\text{(fiber at zero),}\\
F_\infty&=\{([0, 1], [z_0,0,z_2])\} \text{ has size }1& &\text{(fiber at infinity).}
\end{align*}

\subsection{Standard toric action on a Hirzebruch surface}	
\labell{thirz}
Let $S^1$ be identified with the unit circle in $\C$ and consider the Hamiltonian action of $T=(S^1)^2$ on $(\Hirz_n, \eta_\mu)$ induced by
\begin{equation}\labell{toric-act}
(s,t)\cdot ([w_1,w_2],[z_0,z_1,z_2]) = ([w_1,sw_2],[tz_0,z_1,s^n z_2]).
\end{equation}
The moment map image of this torus action is a trapezoid,  as in Figure \ref{trapezoid}, called a \emph{standard Hirzebruch trapezoid} $H_{n,\mu}$. The zero section $S_0$, the section at infinity $S_\infty$, the fiber at zero $F_0$, and the fiber at infinity $F_\infty$ are $T$-invariant spheres, whose moment map images are the top, bottom, left, and right edges of the trapezoid, respectively.
By Delzant's construction \cite{delzant} there is an integrable $T$-invariant almost complex structure $J_T \in \J(\Hirz_n, \eta_\mu)$, for which the spheres $S_0$, $S_\infty$, $F_0$, $F_\infty$ are $J_T$-holomorphic.

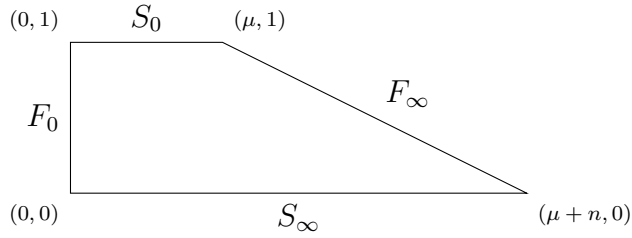
\begin{figure}
\begin{tikzpicture}
	\draw (0,0) node [below left] {\tiny $(0,0)$}
	-- node [below] {$S_\infty$} (6,0) node [below right] {\tiny $(\mu+n, 0)$}
	-- node [above right] {$F_\infty$} (2, 2) node [above right] {\tiny $(\mu, 1)$}
	-- node [above] {$S_0$} (0, 2) node [above left] {\tiny $(0, 1)$}
	-- node [left] {$F_0$} (0,0) -- cycle;
\end{tikzpicture}
\caption{standard Hirzebruch trapezoid}\labell{trapezoid}
\end{figure}

\subsection{Standard cyclic action on a Hirzebruch surface}\labell{hirzcyclic}
Let $G$ be a cyclic group of finite order $r$. Let a generator of $G$ be identified with a primitive $r$th root of unity $\xi$ in $\C$. Consider the $G$-action on the Hirzebruch surface $\Hirz_n$ defined by
\[
\xi\cdot (z, [x_1, x_2]) =
\begin{cases} (\xi^a z, [\xi^b x_1, x_2]) & \text{on } V_1, \\
(\xi^{-a} z, [\xi^b x_1, \xi^{na}x_2]) & \text{on } V_2,
\end{cases}
\] where $\gcd(a,b,r)=1$.
Equivalently,
\begin{equation}\labell{cyclic-act}
\xi \cdot ([w_1, w_2], [z_0, z_1, z_2]) = ([w_1, \xi^a w_2], [\xi^b z_0, z_1, \xi^{na} z_2]).
\end{equation}
We shall call such an action a standard $G$-action on the Hirzebruch surface and refer to a Hirzebruch surface with such a standard $G$-action as $\Hirz_n(a,-b)$.

This $G$-action sends a fiber to a fiber while keeping two fibers invariant: $F_0$ and $F_\infty$. The zero section $S_0$ and the section at infity $S_\infty$ are also $G$-invariant. 
By definition, the integers $(a,-b) \pmod r$ are the rotation numbers of the $G$-action at the fixed point $S_0 \cap F_0$, that is, the weights of the $G$-representation on the tangent space of the fixed point. The rotation numbers at the fixed point $S_\infty \cap F_\infty$ are $(-a, b-na)$.

Moreover, let $c=\gcd(a,b)$ and so $\gcd(c,r)=1$; the standard $G$-action \eqref{cyclic-act} is the restriction induced by the inclusion $G \hookrightarrow S^1$, $g \mapsto g^c$ of the circle action defined in the same way by replacing $\xi$ by $\theta \in S^1$, $a$ by $a/c$, and $b$ by $b/c$, which can also be seen as the circle action obtained as the inclusion $S^1 \hookrightarrow T^2$, $\theta \mapsto (\theta^{a/c}, \theta^{b/c})$ followed by the toric action $\eqref{toric-act}$.

\section{Cyclic actions on the sphere}

\begin{Proposition}\labell{moseractst}
A symplectic action of a cyclic group $G$ of finite order on $(S^2,\tau)$
is conjugate to the restriction of a circle action by a symplectomorphism. 
\end{Proposition}

Proposition \ref{moseractst} follows from well known results. Here is a sketch of the proof.
By Claim \ref{claimginv}, there is a $G$-invariant $\tau$-compatible almost complex structure on $S^2$.
It follows then from the Newlander--Nirenberg Theorem that it must be integrable. Hence,
by the Uniformization Theorem, it is biholomrphic to the standard almost complex structure on $S^2$. So $G$ is a subgroup of the automorphism group $\PSL(2,\C)$ of $\CP^1$.
In $\PSL(2,\C)$ every element is conjugate to a rotation map, and in particular a generator of $G$ is conjugate to a rotation map.
By further applying the equivariant version of Moser's method, noting that the ingredients in the usual Moser's method can be made $G$-equivariant, the symplectic action of the finite cyclic group $G$ on $(S^2,\tau)$ is conjugate by a symplectomorphism to the restriction of a circle action.

\begin{Remark}
	The circle action in the above proposition does not need to be assumed effective. 
	
	Suppose the order of the cyclic group $G$ is $r$, and suppose the weight of the cyclic action at a fixed point on $S^2$ is $a$. By Proposition \ref{moseractst}, up to conjugation, the cyclic action is the restriction induced by the inclusion $G \hookrightarrow S^1$, $g \mapsto g^a$ of the effective circle action $\theta \cdot (z, x) \mapsto (\theta z, x)$, where $\theta \in S^1$ and $(z, x) \in S^2$ identified as the unit sphere in $\R^3 \cong \C \times \R$. On the other hand, this $G$-action can also be considered as a restriction of a non-effective circle action $\theta \cdot (z, x) \mapsto (\theta^a z, x)$ induced by the inclusion $g \mapsto g$. In either case, the integer $a \pmod{r}$ will be called the rotation number (at the chosen fixed point) of the $G$-action on the sphere $S^2$. 
	 
	 Moreover, the cyclic action in the above proposition does not need to be assumed effective either. If the cyclic action is effective, then $\gcd(a, r)=1$.
\hfill $\diamondsuit$ 	   	 
\end{Remark}

\end{document}